\documentclass{amsart}
\usepackage{amsfonts}
\usepackage{amssymb}
\usepackage{times}
\usepackage{xcolor}
\usepackage{amsmath}
\usepackage{amsthm}
\usepackage{comment}
\usepackage{xcolor}

\makeatletter
\def\Ddots{\mathinner{\mkern1mu\raise\p@
\vbox{\kern7\p@\hbox{.}}\mkern2mu
\raise4\p@\hbox{.}\mkern2mu\raise7\p@\hbox{.}\mkern1mu}}
\makeatother

\newtheorem{theorem}{Theorem}[section]
\newtheorem{corollary}{Corollary}[section]
\newtheorem{lemma}{Lemma}[section]
\newtheorem{claim}{Claim}[section]

\newtheorem{definition}{Definition}[section]

\newtheorem{remark}{Remark}[section]

\newenvironment{dem}[1][Proof of Theorem 2.1]{\noindent\textit{#1.} }{\hfill $\square$}
\newenvironment{dem2}[1][Proof of Theorem 2.2]{\noindent\textit{#1.} }{\hfill $\square$}

\def\A{{\mathfrak A}\, }

\begin{document}

\title{$*$-Lie-type maps on alternative $*$-algebras}

\thanks{The first author was supported by the Coordenação de Aperfeiçoamento de Pessoal de Nível Superior - Brasil (CAPES)-Finance 001. The second author was supported by the Centre for Mathematics of the University of Coimbra - UIDB/00324/2020, funded by the Portuguese Government through FCT/MCTES}

\author[Aline Jaqueline de Oliveira Andrade]{Aline Jaqueline de Oliveira Andrade}
\address{Aline Jaqueline de Oliveira Andrade, Federal University of ABC, 
dos Estados Avenue, 5001, 
09210-580, Santo Andr\'{e}, Brazil\\ 
{\em E-mail}: {\tt aline.jaqueline@ufabc.edu.br}}{}

\author[Elisabete~Barreiro]{Elisabete~Barreiro}
\address{Elisabete~Barreiro, University of Coimbra, CMUC, Department of Mathematics, Apartado 3008
EC Santa Cruz
3001 – 501 Coimbra, Portugal\\
{\em E-mail}: {\tt mefb@mat.uc.pt}}{}

\author[Bruno Leonardo Macedo Ferreira]{Bruno Leonardo Macedo Ferreira}
\address{Bruno Leonardo Macedo Ferreira, Federal University of Technology, 
Avenida Professora Laura Pacheco Bastos, 800, 
85053-510, Guarapuava, Brazil\\
{\em E-mail}: {\tt brunolmfalg@gmail.com}}{}

\begin{abstract}
Let $\A$ and $\A'$ be two alternative $*$-algebras with identities $1_{\A}$ and $1_{\A'}$, respectively, and $e_1$ and $e_2 = 1_{\A} - e_1$ nontrivial symmetric idempotents in $\A$. In this paper we study the characterization of multiplicative $*$-Lie-type maps. As application, we get a result on alternative $W^{*}$-algebras.
%In particular, if $\mathcal{M}$ is a factor von Neumann algebra then every complex scalar multiplication bijective unital multiplicative 
%$*$-Lie-type map is $*$-isomorphism. 
%Even more, let $\A$ be an unital alternative $*$-algebra. Assume that $\A$ contains
%a nontrivial symmetric idempotent element $e$ which satisfies $x \A \cdot e = 0$ implies $x = 0$ and
%$x \A \cdot (1_{\A} - e) = 0$ implies $x = 0$. In this paper, it is shown that $\Phi$ is a
%nonlinear $*$-Lie-type derivation on $\A$ if and only if $\Phi$ is an additive
%$*$-derivation. As application, we get a result on alternative $W^{*}$-algebras.

\vspace*{.1cm}

\vspace*{.1cm}

\end{abstract}

\maketitle

\section{Introduction and Preliminaries}

The study of additivity of maps have received a fair amount of attention of mathematicians. The first quite surprising result is due to Martindale who established a condition on a ring such that multiplicative bijective maps are all additive \cite{Mart}. Besides, over the years several works have been published considering different types of associative and non-associative algebras among them we can mention \cite{chang, Fer3, Fer, bruth, Fer1, Fer2, FerGur1}.
In order to add new ingredients to the study of additivity of maps, many researches have devoted themselves to the investigation of two new products, presented by Bre$\check{s}$ar and Fo$\check{s}$ner in \cite{brefos1, brefos2}, where the definition is as follows: for $a, b \in R$, where $R$ is a $*-$ring, we denote by $\{a,b\}_{*} =ab+ba^{*}$ and $[a, b]_{*} = ab-ba^{*}$ the $*$-Jordan product and the $*$-Lie product, respectively. In \cite{LiLuFang}, the authors proved that a map $\varphi$ between two factor von Newmann algebras is a $*$-ring isomorphism if and only if $\varphi(\{a, b\}_{*}) = \{\varphi(a), \varphi(b)\}_{*}$. In \cite{Ferco}, Ferreira and Costa extended these new products and defined two other types of applications, named multiplicative $*$-Jordan n-map and multiplicative $*$-Lie n-map and used it to impose condition such that a map between $C^*$-algebras is a $*$-ring isomorphism.

With this picture in mind, in this article we will discuss when a multiplicative $*$-Lie $n$-map is a $*$-isomorphism in the case of alternative $*$-algebras and, just as it was done in \cite{Ferco2}, we provide an application on alternative $W^{*}$-algebras. Throughout the paper, the ground field is assumed to be the field of complex numbers.

%\begin{definition} \rm

%\end{definition} 

Let $\A$ and $\A'$ be two algebras  with identities $1_{\A}$ and $1_{\A'},$ respectively, and $\varphi : \A \rightarrow \A'$ a map. We have the following concepts:
\begin{itemize}
\item[i.]  $\varphi$ {\it preserves product}  if $\varphi(ab) = \varphi(a)\varphi(b)$, for all elements $a, b \in \A$;
\item[ii.]  $\varphi$ {\it preserves Lie product} if  $\varphi(ab - ba) = \varphi(a)\varphi(b) - \varphi(b)\varphi(a)$, for any $a, b \in \A$;
\item[iii.] $\varphi$ is 
{\it additive}  if $\varphi(a + b) = \varphi(a) + \varphi(b)$, for any $a, b \in \A$;
\item[iv.]  $\varphi$ is   {\it isomorphism}  if $\varphi$ is a bijection additive  that preserves products and scalar multiplication;
\item[v.] $\varphi $ is {\it unital} if $\varphi(1_{\A}) = 1_{\A'}.$
\end{itemize}

%{\color{magenta} Moreover we say $\Phi$ is a isomorphism if $\Phi$ is a map bijective, additive and $\Phi(\alpha x) = \alpha \Phi(x),$ for all $\alpha \in \mathbb{C}$ and $x \in \A.$}

%{\color{magenta} 
%Let $\A$ and $\A'$ be two algebras, with identities $ }

An algebra $ \A$ is called    {\it $*$-algebra}  if $ \A$ is endowed with a involution. By involution, we mean a mapping $* : \A \rightarrow \A$ such that $(x + y)^{*} = x^{*} + y^{*}$, $ (x^{*})^{*} = x$ and $(xy)^{*} = y^{*}x^{*}$ for all $x, y \in \A$. An element $s \in \A$ satisfying $s^{*} = s$ is called   {\it symmetric element}  of $\A$.

Let $\A$ and $\A'$ be two $*$-algebras and $\varphi : \A \rightarrow \A'$ a map. We have the following definitions:
\begin{itemize}
\item[i.] $\varphi$ {\it preserves involution}  if $\varphi(a^{*}) = \varphi(a)^{*}$, for all elements $a \in \A$;
\item[ii.] $\varphi$ is   {\it $*$-isomorphism} if $\varphi$ is an isomorphism that preserves involution;
\item[iii.] $\varphi$ is {\it $*$-additive}  if it preserves involution and it is additive.
\end{itemize}

\begin{definition}[To see \cite{Ferco}] \rm
Consider a  $*$-algebra $\A,$ we denote  $[x_1,x_2]_{*} = x_1x_2-x_2x_{1}^{*},$ for all $x_1,$ $x_2 \in \A$ and the sequence of polynomials,
$$p_{1_*}(x) = x\, \,  \text{and}\, \,  p_{n_*}(x_1, x_2, \ldots , x_n) = \left[p_{(n-1)_*}(x_1, x_2, \ldots , x_{n-1}) , x_n\right]_{*},$$
for all integers $n \geq 2$ and $x_1, \ldots, x_n \in \A.$ 
\end{definition}

\noindent Thus, $p_{2_*}(x_1, x_2) = \left[x_1, x_2\right]_{*} = x_1x_2-x_2x_{1}^{*},$ for all $x_1,$ $x_2 \in \A,$ $p_{3_*} (x_1, x_2, x_3) = \left[\left[x_1, x_2\right]_{*} , x_3\right]_{*},$ for all $x_1,$ $x_2,$ $x_3 \in \A,$ etc. Note that $p_{2_*}$ is the product introduced by Bre$\check{s}$ar and Fo$\check{s}$ner \cite{brefos1, brefos2}. Then, using the nomenclature introduced in \cite{Ferco} 
we have a new class of maps (not necessarily additive).

\begin{definition} \rm
Consider two $*$-algebras $\A$ and $\A'$. A map $\varphi : \A \longrightarrow \A'$ is 
{\it multiplicative $*$-Lie $n$-map} if
\begin{eqnarray*}\label{ident1}
&&\varphi(p_{n_*} (x_1, x_2, \ldots , x_j , \ldots ,x_n)) =  p_{n_*} (\varphi(x_1), \varphi(x_2), \ldots , \varphi(x_j), \ldots ,\varphi(x_n)),
\end{eqnarray*}
for all $x_1,x_2, \ldots, x_n \in \A,$ where $n \geq 2$ is an integer. Multiplicative $*$-Lie $2$-map, $*$-Lie $3$-map 
and $*$-Lie $n$-map are collectively referred to as \textit{multiplicative $*$-Lie-type maps}.
\end{definition}

%A nonlinear $*$-Lie $n$-derivation is a map $\Phi : \A \rightarrow \A$ satisfying the condition
%\begin{eqnarray}\label{ouro}
%\Phi(p_{n_*} (x_1, x_2, . . . , x_k , . . . ,x_n)) =
%\sum_{k=1}^n p_{n_*} (x_1, x_2, . . . , \Phi(x_k) , . . . ,x_n)
%\end{eqnarray}
%for all $x_1, x_2, \ldots x_n \in \A$.

An algebra $\A$ (not necessarily associative or commutative) is called   {\it alternative
algebra}  if it satisfies the identities $a^2b = a(ab)$ and $ba^2 = (ba)a$, for all elements $a, b \in \A$. One easily sees that any associative algebra is an alternative algebra. An alternative algebra $\A$ is called {\it prime}  if for any elements $a, b \in \A$ satisfying the condition $a\A b = 0$, then either $a = 0$ or $b = 0$.

We consider an alternative algebra $\mathfrak{A}$ with identity $1_{\A}.$ Fix a nontrivial idempotent element $e_{1}\in\mathfrak{A}$ and denote $e_2 = 1_{\A} - e_1.$ It is easy to see that $(e_ka) e_j=e_k(ae_j)~(k,j=1,2)$ for all $a\in \mathfrak{A}$. Then $\mathfrak{A}$ has a Peirce decomposition
$$\mathfrak{A}=\mathfrak{A}_{11}\oplus \mathfrak{A}_{12}\oplus
\mathfrak{A}_{21}\oplus \mathfrak{A}_{22},$$ where
$\mathfrak{A}_{kj}:=e_{k}\mathfrak{A}e_{j}$ $(k,j=1,2)$ (see \cite{He}), satisfying the following multiplicative relations:
\begin{enumerate}\label{asquatro}
\item [\it (i)] $\mathfrak{A}_{kj}\mathfrak{A}_{jl}\subseteq\mathfrak{A}_{kl}\
(k,j,l=1,2);$
\item [\it (ii)] $\mathfrak{A}_{kj}\mathfrak{A}_{kj}\subseteq \mathfrak{A}_{jk}\
(k,j=1,2);$
\item [\it (iii)] $\mathfrak{A}_{kj}\mathfrak{A}_{ml}=\{0\},$ if $j\neq m$ and
$(k,j)\neq (m,l),\ (k,j,m,l=1,2);$
\item [\it (iv)] $x_{kj}^{2}=0,$ for all $x_{kj}\in \mathfrak{A}_{kj}\ (k,j=1,2;~k\neq j).$

\end{enumerate}

\section{Main theorem}

In the following we shall prove a part of the main result of this paper.

\begin{theorem}\label{mainthm1} 
Let $\A$ and $\A'$ be two alternative  $*$-algebras with identities $1_{\A}$ and $1_{\A'},$ respectively, and $e_1$ and $e_2 = 1_{\A} - e_1$ nontrivial symmetric idempotents in $\A$. Suppose that $\A$ satisfies
\begin{equation}\label{spadesuit}
(e_j\A) x = \left\{0\right\} \mbox{ for any }  j \in \{ 1,2 \} \ \ \  \mbox{ implies } \ \ \ x = 0
\end{equation}

\noindent Suppose also that $\varphi: \A \rightarrow \A'$ is a multiplication bijective unital map which satisfies
\begin{equation}\label{bullet}
\varphi(p_{n_*}(a,b,\xi,\ldots,\xi)) = p_{n_*}(\varphi(a),\varphi(b),\varphi(\xi),\ldots,\varphi(\xi)),\  
\end{equation}
for all $a, b \in \A$ and $\xi \in \left\{e_1, e_2,1_{\A}\right\}$. Then $\varphi$ is $*$-additive.
\end{theorem}

\noindent The following claims and lemmas have the same hypotheses as the Theorem \ref{mainthm1} and
we need them to prove the $*$-additivity of $\varphi$. 
\begin{claim}\label{claim1}
$*(\A_{kj})\subset \A_{jk}$, for $j,k\in \{1,2\}$.
\end{claim}
\begin{proof}
If $a_{kj} \in \A_{kj}$ then $$a^*_{kj} = (e_k a_{kj} e_j)^* = (e_j)^*(a_{kj})^* (e_k)^* = e_j (a_{kj})^* e_k \in \A_{jk}.$$
\end{proof}

\noindent It is easy to check the following result (see \cite{Ferco2}).
\begin{claim}\label{claim2}
Let $x,y,h$  in $\A$ such that $\varphi(h) = \varphi(x) + \varphi(y)$. Then,
 given $z \in \A$,
$$\varphi(p_{n_*}(h,z,\xi, \ldots,\xi)) = \varphi(p_{n_*}(x,z,\xi, \ldots,\xi))
+ \varphi(p_{n_*}(y,z,\xi,\ldots,\xi))$$
and
$$
\varphi(p_{n_*}(z,h,\xi,\ldots,\xi)) = \varphi(p_{n_*}(z,x,\xi,\ldots,\xi)) 
                                                   + \varphi(p_{n_*}(z,y,\xi,\ldots\xi))$$
for $\xi \in \left\{e_1, e_2,1_{\A}\right\}$.

%$$
%\5begin{aligned}
%\varphi(q_{n*}(T_1,...,T_{i-1},H,T_{i+1},...,T_n)) &= \varphi(q_{n*}(T_1,...,T_{i-1},X,T_{i+1},...,T_n)) \\
%                                                   &+ \varphi(q_{n*}(T_1,...,T_{i-1},Y,T_{i+1},...,T_n))
%\end{aligned}
%$$
%for $T_j = I_\A$ or $T_j = P$ with $j = 1,...,n$.
\end{claim}

\begin{claim}\label{claim3}  $\varphi(0) = 0$.
\end{claim}
\begin{proof}
Since $\varphi$ is surjective, there exists $x \in \A$ such that $\varphi(x) = 0$. Then,
$$ 
\begin{aligned}
\varphi(0) &= \varphi(p_{n_*}(0,x,1_{\A},\ldots,1_{\A})) = p_{n_*}(\varphi(0),\varphi(x),\varphi(1_{\A}),\ldots,\varphi(1_{\A}))  \\
                                         &=p_{n_*}(\varphi(0),0,\varphi(1_{\A}),\ldots,\varphi(1_{\A})) = 0.
\end{aligned}
$$
\end{proof}

\noindent The next results aim to show the additivity of $\varphi$.

\begin{lemma}\label{lema1} For any $a_{11} \in \A_{11}$ and $b_{22} \in \A_{22}$, we have 
$$\varphi(a_{11} + b_{22}) = \varphi(a_{11}) + \varphi(b_{22}).$$
\end{lemma}
\begin{proof}
Since $\varphi$ is surjective, given $\varphi(a_{11})+\varphi(b_{22}) \in \A'$ there exists $h \in \A$ such that 
$\varphi(h) = \varphi(a_{11})+\varphi(b_{22})$. We may write  $h=h_{11}+h_{12}+h_{21}+h_{22}$, with $ h_{jk} \in \mathfrak{A}_{jk}\
(k,j=1,2)$. Besides, by Claims \ref{claim2} and \ref{claim3}
$$
\varphi(p_{n_*}(e_1,h,e_1,\ldots,e_1)) = \varphi(p_{n_*}(e_1,a_{11},e_1,\ldots,e_1)) + \varphi(p_{n_*}(e_1,b_{22},e_1,\ldots,e_1)), 
$$
that is, 
$$\varphi(-h_{21} + h_{21}^*)= \varphi(0) + \varphi(0) = 0.$$
Then, by injectivity of $\varphi$, $-h_{21} + h_{21}^* = 0$. Thus $h_{21} = 0$. Moreover, 
$$
\varphi(p_{n_*}(e_2,h,e_2,\ldots,e_2)) = \varphi(p_{n_*}(e_2,a_{11},e_2,\ldots,e_2)) + \varphi(p_{n_*}(e_2,b_{22},e_2,\ldots,e_2)), 
$$
that is,
$$\varphi(-h_{12} + h_{12}^*)= 0.$$
Again, by injectivity of $\varphi$ we conclude that $h_{12} = 0$.

Furthermore, given $d_{21}\in \A_{21}$,
\[
\varphi(p_{n_*}(d_{21},h,e_1,\ldots,e_1))=\varphi(p_{n_*}(d_{21},a_{11},e_1,\ldots,e_1)) + \varphi(p_{n_*}(d_{21},b_{22},e_1,\ldots,e_1)),
\]
%$$ 
%\begin{aligned}
%\varphi(p_{n_*}(d_{21},h,e_1,\ldots,e_1)) &=\varphi(p_{n_*}(d_{21},a_{11},e_1,\ldots,e_1)) \\
%&+\varphi(p_{n_*}(d_{21},b_{22},e_1,\ldots,e_1)), 
%\end{aligned}$$
that is,
\[
\varphi(d_{21}h_{11} - (d_{21}h_{11})^*) = \varphi(d_{21}a_{11} - (d_{21}a_{11})^*).
\]
Then we conclude, by injectivity of $\varphi$, that $d_{21}h_{11} - (d_{21}h_{11})^* = d_{21}a_{11} - (d_{21}a_{11})^*$, that is, 
$d_{21}(h_{11} - a_{11}) = 0$. Even more, $(e_2\A)(h_{11} - a_{11}) = 0$, which implies that $h_{11} = a_{11}$ by   Condition \eqref{spadesuit} of Theorem \ref{mainthm1}.

Finally, given $d_{12}\in \A_{12}$, a similar calculation gives us $h_{22} = b_{22}$. Therefore $h = a_{11} + b_{22}$.
\end{proof}

\begin{lemma}\label{lema2}
For any $a_{12} \in \A_{12}$ and $b_{21} \in \A_{21}$, we have $\varphi(a_{12} + b_{21}) = \varphi(a_{12}) + \varphi(b_{21})$.
\end{lemma}

\begin{proof}
Since $\varphi$ is surjective, given $\varphi(a_{12})+\varphi(b_{21}) \in \A'$ there exists $h \in \A$ such that 
$\varphi(h) = \varphi(a_{12})+\varphi(b_{21})$. We may write  $h=h_{11}+h_{12}+h_{21}+h_{22}$, with $ h_{jk} \in \mathfrak{A}_{jk}\
(k,j=1,2)$. Now, by Claims \ref{claim2} and \ref{claim3}
$$
\varphi(p_{n_*}(e_1,h,e_1,\ldots,e_1)) = \varphi(p_{n_*}(e_1,a_{12},e_1,\ldots,e_1)) + \varphi(p_{n_*}(e_1,b_{21},e_1,\ldots,e_1)),
$$
that is,
$$
\varphi(-h_{21} + h_{21}^*) = \varphi(-b_{21} + b_{21}^*).
$$
Then, by injectivity of $\varphi$, $-h_{21} + h_{21}^* = -b_{21} + b_{21}^*$. Thus $h_{21} = b_{21}$. Moreover,
$$
\varphi(p_{n_*}(e_2,h,e_2,\ldots,e_2)) = \varphi(p_{n_*}(e_2,a_{12},e_2,\ldots,e_2)) + \varphi(p_{n_*}(e_2,b_{21},e_2,\ldots,e_2)),$$
that is,
$$
\varphi(-h_{12} + h_{12}^*) = \varphi(-a_{12} + a_{12}^*).
$$
Again, by injectivity of $\varphi$ we conclude that $h_{12} = a_{12}$.

Furthermore, given $d_{21}\in \A_{21}$,
$$ 
\begin{aligned}
\varphi(d_{21}h_{11} - (d_{21}h_{11})^*) &= \varphi(p_{n_*}(d_{21},h,e_1,\ldots,e_1)) \\
                                         &= \varphi(p_{n_*}(d_{21},a_{12},e_1,\ldots,e_1))  + \varphi(p_{n_*}(d_{21},b_{21},e_1,\ldots,e_1)) = 0.
\end{aligned}
$$
Then we conclude, by injectivity of $\varphi$, that $d_{21}h_{11} - (d_{21}h_{11})^* = 0$, that is, $d_{21}h_{11} = 0$. Even more, $(e_2\A) h_{11} = 0$, 
which implies that $h_{11} = 0$   by   Condition \eqref{spadesuit} of Theorem \ref{mainthm1}.

Finally, given $d_{12}\in \A_{12}$, a similar calculation gives us $h_{22} = 0$. Therefore, we conclude that $h = a_{12} +b_{21}$.
\end{proof}

\begin{lemma}\label{lema3}
For any $a_{11} \in \A_{11}$, $b_{12} \in \A_{12}$, $c_{21} \in \A_{21}$ and $d_{22} \in \A_{22}$ we have 
$$\varphi(a_{11} + b_{12} + c_{21} + d_{22}) = \varphi(a_{11}) + \varphi(b_{12}) + \varphi(c_{21}) + \varphi(d_{22}).$$
\end{lemma}
\begin{proof}
Since $\varphi$ is surjective, given $\varphi(a_{11})+\varphi(b_{12})+\varphi(c_{21}) + \varphi(d_{22})\in \A'$ there exists $h \in \A$ such that 
$\varphi(h) = \varphi(a_{11})+\varphi(b_{12})+\varphi(c_{21}) + \varphi(d_{22})$. We may write  $h=h_{11}+h_{12}+h_{21}+h_{22}$, with $ h_{jk} \in \mathfrak{A}_{jk}\
(k,j=1,2)$. Applying  Lemmas \ref{lema1} and \ref{lema2} we have
$$
\varphi(h) = \varphi(a_{11})+\varphi(b_{12})+\varphi(c_{21}) + \varphi(d_{22}) = \varphi(a_{11} + d_{22})+\varphi(b_{12} + c_{21}).
$$
Now, observing that
$p_{n_*}(e_1,a_{11} + d_{22},e_1,\ldots,e_1) = 0 = p_{n_*}(e_1,b_{12},e_1,\ldots,e_1)$ and by Claims \ref{claim2} and \ref{claim3} we obtain
$$
\begin{aligned}
&\varphi(p_{n_*}(e_1,h,e_1,\ldots,e_1)) \\
&= \varphi(p_{n_*}(e_1,a_{11} + d_{22},e_1,\ldots,e_1)) + \varphi(p_{n_*}(e_1,b_{12} + c_{21},e_1,\ldots,e_1)) \\
&= \varphi(p_{n_*}(e_1,c_{21},e_1,\ldots,e_1)), 
\end{aligned}
$$
that is, 
$$
\varphi(-h_{21} + h_{21}^*) = \varphi(-c_{21} + c_{21}^*).
$$
Then, by injectivity of $\varphi$, $-h_{21} + h_{21}^* = -c_{21} + c_{21}^*$. Thus $h_{21} = c_{21}$.

In a similar way, using $e_2$ rather than $e_1$ in the previous calculation, we conclude that $h_{12} = b_{12}$. Also, given $x_{21} \in \A_{21}$,
\[\begin{aligned}
&\varphi(p_{n_*}(x_{21},h,e_1,\ldots,e_1)) \\
&= \varphi(p_{n_*}(x_{21},a_{11} + d_{22},\ldots,e_1)) + \varphi(p_{n_*}(x_{21},b_{12} + c_{21},e_1,\ldots,e_1)) \\
&= \varphi(p_{n_*}(x_{21},a_{11},e_1,\ldots,e_1)), 
\end{aligned}
\]
since $p_{n_*}(x_{21},b_{12} + c_{21},e_1,\ldots,e_1) = 0 = p_{n_*}(x_{21},d_{22},\ldots,e_1)$. 
Again, by injectivity of $\varphi$ we conclude, by following the same strategy as in the proof of Lemma \ref{lema1}, that $h_{11} = a_{11}$. 
Now, using $e_2$ rather than $e_1$ and $x_{12}$ rather than $x_{21}$ in the previous calculation we obtain $h_{22} = d_{22}$. 
Therefore, $h = a_{11} + b_{12} + c_{21} + d_{22}$. 
\end{proof}

\begin{lemma}\label{lema4}
For any $a_{jk}, b_{jk} \in \A_{jk}$, with $j \neq k$, we have $\varphi(a_{jk} + b_{jk}) = \varphi(a_{jk}) + \varphi(b_{jk})$.
\end{lemma}
\begin{proof}
We shall prove the case $j=1$ and $k = 2$. The other case is done in a similar way. Since $\varphi$ is surjective, given 
$\varphi(a_{12}) + \varphi(b_{12})\in \A'$ and $\varphi(-a_{12}^*) + \varphi(-b_{12}^*)$ there exist $h \in \A$ and $t \in \A$ such that 
$\varphi(h) = \varphi(a_{12}) + \varphi(b_{12})$ and $\varphi(t) = \varphi(-a_{12}^*) + \varphi(-b_{12}^*)$.  We may write  $h=h_{11}+h_{12}+h_{21}+h_{22}$
and $t = t_{11} + t_{12} + t_{21} + t_{22}$, with $ h_{jk},t_{jk} \in \mathfrak{A}_{jk}\
(k,j=1,2)$.

First we show that $h\in \A_{12}$. By Claim \ref{claim2} we get
$$
\begin{aligned}
\varphi(-h_{21} + h_{21}^*) &= \varphi(p_{n_*}(e_1,h,e_1,\ldots,e_1)) \\
                            &= \varphi(p_{n_*}(e_1,a_{12},e_1,\ldots,e_1)) + \varphi(p_{n_*}(e_1,b_{12},e_1,\ldots,e_1)) = 0.
\end{aligned}
$$
Then, by injectivity of $\varphi$ we obtain $h_{21} = 0$. Also, given $d_{12}\in \A_{12}$,
$$
\begin{aligned}
\varphi(d_{12}h_{22} - (d_{12}h_{22})^*) &= \varphi(p_{n_*}(d_{12},h,e_2,\ldots,e_2)) \\
                                         &= \varphi(p_{n_*}(d_{12},a_{12},e_2,\ldots,e_2))  + \varphi(p_{n_*}(d_{12},b_{12},e_2,\ldots,e_2)) = 0,
\end{aligned}
$$
that is, $d_{12}h_{22} = 0$, which implies that $h_{22} = 0$ by   Condition \eqref{spadesuit} of Theorem \ref{mainthm1}. Now, using $d_{21}\in \A_{21}$ rather than $d_{12}$ in the previous calculation, 
we conclude that $h_{11} = 0$. Therefore, $h = h_{12}\in \A_{12}$.

In a similar way, we obtain $t = t_{21}\in \A_{21}$.
Finally, by Lemma \ref{lema3}

$$
\begin{aligned}
\varphi(a_{12} + b_{12} - a_{12}^* - b_{12}^*) &= \varphi(p_{n_*}(e_1 + a_{12},e_2 + b_{12},e_2,\ldots,e_2)) \\
                                               &= p_{n_*}(\varphi(e_1 + a_{12}),\varphi(e_2 + b_{12}),\varphi(e_2),\ldots,\varphi(e_2)) \\
																							 &= p_{n_*}(\varphi(e_1),\varphi(e_2),\varphi(e_2),\ldots,\varphi(e_2))  \\
																							 &+ p_{n_*}(\varphi(e_1),\varphi(b_{12}),\varphi(e_2),\ldots,\varphi(e_2)) \\
																							 &+ p_{n_*}(\varphi(a_{12}),\varphi(e_2),\varphi(e_2),\ldots,\varphi(e_2)) \\
																							 &+ p_{n_*}(\varphi(a_{12}),\varphi(b_{12}),\varphi(e_2),\ldots,\varphi(e_2)) \\
																							 &= \varphi(p_{n_*}(e_1,e_2,e_2,\ldots,e_2)) \\
																							 &+ \varphi(p_{n_*}(e_1,b_{12},e_2,\ldots,e_2)) \\
																							 &+ \varphi(p_{n_*}(a_{12},e_2,e_2,\ldots,e_2)) \\
																							 &+ \varphi(p_{n_*}(a_{12},b_{12},e_2,\ldots,e_2)) \\
																							 &= \varphi(a_{12} - a_{12}^*) + \varphi(b_{12} - b_{12}^*) \\
																							 &= \varphi(a_{12}) + \varphi(b_{12}) + \varphi(-a_{12}^*) + \varphi(-b_{12}^*) \\
																							 &= \varphi(h_{12}) + \varphi(t_{21}) = \varphi(h_{12} + t_{21}).																						
\end{aligned}
$$

\noindent Since $\varphi$ is injective, we have $a_{12} + b_{12} - a_{12}^* - b_{12}^* = h_{12} + t_{21}$, this is, $h = h_{12} = a_{12} + b_{12}$.
\end{proof}

\begin{lemma}\label{lema5}
For any $a_{jj}, b_{jj} \in \A_{jj}$,  with $j \in \left\{1,2\right\}$, we have $\varphi(a_{jj} + b_{jj}) = \varphi(a_{jj}) + \varphi(b_{jj})$.
\end{lemma}
\begin{proof}
We shall prove the case $j=1$, since the other case is done in a similar way. Since $\varphi$ is surjective, given 
$\varphi(a_{11}) + \varphi(b_{11})\in \A'$ there exists $h \in \A$ such that 
$\varphi(h) = \varphi(a_{11}) + \varphi(b_{11})$. We may write  $h=h_{11}+h_{12}+h_{21}+h_{22}$, with $ h_{jk} \in \mathfrak{A}_{jk}\
(k,j=1,2)$. Now, by Claim \ref{claim2}
$$
\begin{aligned}
\varphi(-h_{21} + h_{21}^*) &= \varphi(p_{n_*}(e_1,h,e_1,\ldots,e_1)) \\
                            &= \varphi(p_{n_*}(e_1,a_{11},e_1,\ldots,e_1)) + \varphi(p_{n_*}(e_1,b_{11},e_1,\ldots,e_1)) = 0.
\end{aligned}
$$
Then, by injectivity of $\varphi$ we obtain $h_{21} = 0$. Also,
$$
\begin{aligned}
\varphi(-h_{12} + h_{12}^*) &= \varphi(p_{n_*}(e_2,h,e_2,\ldots,e_2)) \\
                            &= \varphi(p_{n_*}(e_2,a_{11},e_2,\ldots,e_2)) + \varphi(p_{n_*}(e_2,b_{11},e_2,\ldots,e_2)) = 0,
\end{aligned}
$$
that is, $h_{12} = 0$ by injectivity of $\varphi$.  Moreover, given $d_{12}\in \A_{12}$,
$$
\begin{aligned}
\varphi(d_{12}h_{22} - (d_{12}h_{22})^*) &= \varphi(p_{n_*}(d_{12},h,e_2,\ldots,e_2)) \\
&= \varphi(p_{n_*}(d_{12},a_{11},e_2,\ldots,e_2)) + \varphi(p_{n_*}(d_{12},b_{11},e_2,\ldots,e_2)) \\
&= 0.
\end{aligned}
$$ 
Then, by injectivity of $\varphi$, $d_{12}h_{22}=0$, which implies that $h_{22} = 0$ by   Condition \eqref{spadesuit} of Theorem \ref{mainthm1}. Finally, given $d_{21}\in \A_{21}$, 
by Lemmas \ref{lema3} and \ref{lema4} we have
$$
\begin{aligned}
\varphi(d_{21}h_{11} - (d_{21}h_{11})^*) 
&= \varphi(p_{n_*}(d_{21},h,e_1,\ldots,e_1)) \\								&= \varphi(p_{n_*}(d_{21},a_{11},e_1,\ldots,e_1)) + \varphi(p_{n_*}(d_{21},b_{11},e_1,\ldots,e_1)) \\
&= \varphi(d_{21}a_{11} - (d_{21}a_{11})^*) + \varphi(d_{21}b_{11} - (d_{21}b_{11})^*) \\
&= \varphi(d_{21}a_{11}) + \varphi(-(d_{21}a_{11})^*)+ \varphi(d_{21}b_{11}) + \varphi(-(d_{21}b_{11})^*) \\		&= \varphi(d_{21}a_{11} + d_{21}b_{11})+ \varphi(-(d_{21}a_{11})^*-(d_{21}b_{11})^*) \\
																				 &= \varphi(d_{21}(a_{11} + b_{11}) - (a_{11}^* + b_{11}^*)d_{21}^*),
\end{aligned}
$$ 
that is, $d_{21}h_{11} - (d_{21}h_{11})^* = d_{21}(a_{11} + b_{11}) - (a_{11}^* + b_{11}^*)d_{21}^*$, by injectivity of $\varphi$. Thus, 
$d_{21}(h_{11} - (a_{11} + b_{11})) = 0$, which implies that $h_{11} = a_{11} + b_{11}$ by   Condition \eqref{spadesuit} of Theorem \ref{mainthm1}.
\end{proof}

%Now we are able to show that $\varphi$ is $*$-additive. Using Lemmas \ref{lema3}, \ref{lema4} and \ref{lema5} we have, for all $a,b \in \A$,
%$$
%\begin{aligned}
%\varphi(a + b) &= \varphi(a_{11}+a_{12}+a_{21}+a_{22}+b_{11}+b_{12}+b_{21}+b_{22}) \\
%               &= \varphi(a_{11}+b_{11})+\varphi(a_{12}+b_{12})+\varphi(a_{21}+b_{21})+\varphi(a_{22}+b_{22}) \\
%               &= \varphi(a_{11})+\varphi(b_{11})+\varphi(a_{12})+\varphi(b_{12}) \\
%							 &+\varphi(a_{21})+\varphi(b_{21})+\varphi(a_{22})+\varphi(b_{22}) \\
%               &= \varphi(a_{11}+a_{12}+a_{21}+a_{22}) \\
%							 &+ \varphi(b_{11}+b_{12}+b_{21}+b_{22}) = \varphi(a) + \varphi(b).
%\end{aligned}
%$$

\begin{dem}
Now using Lemmas~\ref{lema3}, \ref{lema4} and \ref{lema5} is easy see that $\varphi$ is additive. Besides, using additivity of $\varphi$  and since $\varphi$ is unital,  we have for $a \in \A$,
%by additivity of $\varphi$ and since $\varphi$ is unital, we have
$$
\begin{aligned}
2^{n-2}(\varphi(a) - \varphi(a)^*) &= p_{n_*}(\varphi(a), 1_{\A'},\ldots,1_{\A'}) = p_{n_*}(\varphi(a), \varphi(1_{\A}),\ldots,\varphi(1_{\A})) \\
& =\varphi(p_{n_*}(a,1_{\A},\ldots,1_{\A})) = \varphi(2^{n-2}(a - a^*)) \\
& = 2^{n-2}\varphi(a - a^*) = 2^{n-2}(\varphi(a) - \varphi(a^*)),
\end{aligned}
$$
then $\varphi(a^*) = \varphi(a)^*$ and we conclude that $\varphi$ preserves involution.
\end{dem}

%This completes the proof of Theorem \ref{mainthm1}. 

\begin{remark}
Observe that the Theorem~\ref{mainthm1} holds for any field of characteristic different of $2.$ In the proof the Theorem~\ref{mainthm1} we established the additivity of $\varphi$ without using the unital assumption of $\varphi$.
\end{remark}

\begin{theorem}\label{mainthm2} 
Let $\A$ and $\A'$ be two alternative $*$-algebras with identities $1_{\A}$ and $1_{\A'}$, respectively, and $e_1$ and $e_2 = 1_{\A} - e_1$ nontrivial symmetric idempotents in $\A$. Let $\varphi: \A \rightarrow \A'$ be a complex scalar multiplication bijective unital map. Suppose that $\A$ satisfies the conditions of the Theorem~\ref{mainthm1}, namely,
\begin{equation*}
    (e_j\A) x = \left\{0\right\} \mbox{ for any }  j \in \{ 1,2 \} \ \ \  \mbox{ implies } \ \ \ x = 0,
\end{equation*}
\begin{equation*}
    \varphi(p_{n_*}(a,b,\xi,\ldots,\xi)) = p_{n_*}(\varphi(a),\varphi(b),\varphi(\xi),\ldots,\varphi(\xi)),
\end{equation*}
for all $a, b \in \A$ and $\xi \in \left\{e_1, e_2,1_{\A}\right\}$.
%\begin{eqnarray*}
%  &&\left(\spadesuit\right) \ \ \  \ \ \  (e_j\A) x = \left\{0\right\} \ \ \  \mbox{implies} \ \ \ x = 0.	
%\end{eqnarray*}

Even more, if  $\A'$ satisfies the condition
\begin{equation}\label{clubsuit}
    (\varphi(e_j)\A') y = \left\{0\right\} \mbox{ for any }  j \in \{ 1,2 \} \ \ \  \mbox{implies} \ \ \ y = 0,
\end{equation}
%\begin{eqnarray*}
% &&\left(\clubsuit\right) \ \ \  \ \ \  (\varphi(e_j)\A') y = \left\{0\right\} \ \ \  \mbox{implies} \ \ \ y = 0,\\
% and && \\
% &&\left(\bullet\right)\varphi(p_{n_*}(a,b,\xi,...,\xi)) = p_{n_*}(\varphi(a),\varphi(b),\varphi(\xi),...,\varphi(\xi)),
%\end{eqnarray*}
then $\varphi$ is $*$-isomorphism.
\end{theorem}

With this hypothesis and Theorem~\ref{mainthm1} we have already proved that $\varphi$ is $*$-additive. It remains for us to show that $\varphi$ preserves product. In order to do that we will prove some more lemmas. 
Firstly, we observe that, 
%for any $a \in \A$,
%\begin{remark}
%Denoting by $\mathcal{R}(a)$ and $\mathcal{I}(a)$ the real part and imaginary part of $a \in \A$, respectively, we have
%$$
%\begin{aligned}
%\varphi(2^{n-1}\mathcal{R}(a)) &= \varphi(2^{n-2}(a + a^*)) = \varphi(p_{n_*}(ia,-i1_{\A},1_{\A},...,1_{\A})) \\
 %                              &= p_{n_*}(i\varphi(a),-i\varphi(1_{\A}),1_{\A'},...,1_{\A'}) = 2^{n-2}(\varphi(a) + \varphi(a)^*) \\
															 %&= 2^{n-1}\mathcal{R}(\varphi(a))
%\end{aligned}															
%$$
%\noindent and
%$$
%\begin{aligned}
%\varphi(2^{n-1}i\, \mathcal{I}(a)) &= \varphi(2^{n-2}(a - %a^*)) = \varphi(p_{n_*}(a,1_{\A},...,1_{\A})) \\
 %                                   &= p_{n_*}(\varphi(a),1_{\A'},...,1_{\A'}) = 2^{n-2}(\varphi(a) - \varphi(a)^*) \\
															 %     &= 2^{n-1}i\,\mathcal{I}(\varphi(a)).
%\end{aligned}
%$$ 
%\end{remark}

%Even more,
\begin{claim}\label{claim4}
$q_j = \varphi(e_j)$ is an idempotent in $\A'$, for $j\in \{1,2\}$.
\end{claim}
\begin{proof}
Since $\varphi$ is a complex scalar multiplication, it follows that
$$
\begin{aligned}
2^{n-1}i\,q_j &= 2^{n-1}i\varphi(e_j) = \varphi(2^{n-1}ie_j) = \varphi(p_{n_*}(ie_j,e_j,1_{\A},\ldots,1_{\A})) \\
&= p_{n_*}(i\varphi(e_j),\varphi(e_j), \varphi(1_{\A}), \ldots, \varphi(1_{\A})) \\
&= p_{n_*}(i\varphi(e_j),\varphi(e_j),1_{\A'},\ldots,1_{\A'}))= 2^{n-1}i\varphi(e_j)^2 = 2^{n-1}i\,{q_j}^2.                                                          
\end{aligned}
$$
Then we can conclude that $q_j = {q_j}^2$. Moreover, since $e_j$ is a idempotent in $\A$ we have that $p_{n_*}(e_j,1_{\A},\ldots,1_{\A}) = 0$. Besides,
$$
0 = \varphi(0) = \varphi(p_{n_*}(e_j,1_{\A},\ldots,1_{\A})) = p_{n_*}(q_j,1_{\A'},\ldots,1_{\A'}). 
$$
Thus, $q_j - {q_j}^*=0$, that is, $q_j = {q_j}^*$.
\end{proof}

\begin{lemma}\label{lemanovo}
For any $a\in \A$, $\varphi(e_ja) = \varphi(e_j)\varphi(a)$ and $\varphi(ae_j) = \varphi(a)\varphi(e_j)$, with  $j\in \{1,2\}$.
\end{lemma}
\begin{proof}
Firstly, observe that
$$
p_{n_*}(ia,e_j,1_{\A},\ldots,1_{\A}) = 2^{n-2}i(ae_j + e_ja^*)
$$
and
$$p_{n_*}(a,e_j,1_{\A},\ldots,1_{\A}) = 2^{n-2}(ae_j -e_ja^*).$$
Still, by Condition $(\ref{bullet})$ of Theorem~\ref{mainthm1} and $*$-additivity of $\varphi$,
$$
\begin{aligned}
\varphi(2^{n-2}i(ae_j + e_ja^*)) &= \varphi(p_{n_*}(ia,e_j,1_{\A},\ldots,1_{\A})) = p_{n_*}(\varphi(ia),\varphi(e_j),1_{\A'},\ldots,1_{\A'}) \\
&= 2^{n-2}i(\varphi(a)\varphi(e_j) + \varphi(e_j)\varphi(a)^*)
\end{aligned}
$$
and
$$
\begin{aligned}
\varphi(2^{n-2}(ae_j - e_ja^*)) &= \varphi(p_{n_*}(a,e_j,1_{\A},\ldots,1_{\A})) = p_{n_*}(\varphi(a),\varphi(e_j),1_{\A'},\ldots,1_{\A'}) \\
&= 2^{n-2}(\varphi(a)\varphi(e_j) - \varphi(e_j)\varphi(a)^*).
\end{aligned}
$$
Now, since $\varphi$ is $*$-additive, multiplying the second equality by $i$ and adding these two equations we obtain $\varphi(ae_j) = \varphi(a)\varphi(e_j)$.
The second statement is obtained in a similar way.
\end{proof}

Consider the Peirce decomposition of $\A'$ with respect to idempotents $q_j=\varphi(e_j)$ of $\A'$ (with $j\in \{1,2\}$)   given by $\A' = \A'_{11} \oplus \A'_{12} \oplus \A'_{21} \oplus \A'_{22},$ where $\A'_{kj} := q_k\A'q_j $ for $k,j\in \{1,2\}$.

\begin{lemma}\label{lemadireto}
$\varphi(\A_{jk}) \subset \A_{jk}'$, for $j,k\in \{1,2\}$.
\end{lemma}
\begin{proof}
Given $x\in \A_{jk}$, we have $x = e_jxe_k$ and then, by Lemma \ref{lemanovo}, 
$$\varphi(x) = \varphi(e_j)\varphi(xe_k) = \varphi(e_j)\varphi(x)\varphi(e_k) \in \A_{jk}'.$$
\end{proof}

\begin{lemma}\label{lemam2}
For $j\neq k$, we have:
\begin{itemize}
\item If $a_{jk} \in \A_{jk}$ and $b_{kk} \in \A_{kk}$ then $\varphi(a_{jk}b_{kk})=\varphi(a_{jk})\varphi(b_{kk})$;
\item If $a_{jk} \in \A_{jk}$ and $b_{jk} \in \A_{jk}$ then $\varphi(a_{jk}b_{jk})=\varphi(a_{jk})\varphi(b_{jk})$;
\item If $a_{jj} \in \A_{jj}$ and $b_{jk} \in \A_{jk}$ then $\varphi(a_{jj}b_{jk})=\varphi(a_{jj})\varphi(b_{jk})$;
\item If $a_{jk} \in \A_{jk}$ and $b_{kj} \in \A_{kj}$ then $\varphi(a_{jk}b_{kj})=\varphi(a_{jk})\varphi(b_{kj})$.
\end{itemize}
\end{lemma}
\begin{proof}
In order to prove the first statement, on the one hand, by Lemma \ref{lemadireto}
$$
\begin{aligned}
\varphi(a_{jk}b_{kk}) - \varphi(a_{jk}b_{kk})^* 
&= \varphi(a_{jk}b_{kk} - (a_{jk}b_{kk})^*)=\varphi(p_{n_*}(a_{jk},b_{kk},e_k,\ldots,e_k)) \\ 
&= p_{n_*}(\varphi(a_{jk}), \varphi(b_{kk}), q_k,\ldots,q_k)\\
&= \varphi(a_{jk})\varphi(b_{kk}) - (\varphi(a_{jk})\varphi(b_{kk}))^*
\end{aligned}
$$
and then $\varphi(a_{jk}b_{kk}) = \varphi(a_{jk})\varphi(b_{kk})$.

Now to prove the second statement, we have
\[
\begin{aligned}
\varphi(a_{jk}b_{jk}) & - \varphi(a_{jk}b_{jk})^*  -2^{n-3} \varphi( b_{jk} a_{jk})^{*} + 2^{n-3}\varphi(b_{jk} a_{jk}^*)^* \\
&= \varphi(a_{jk}b_{jk}) - (a_{jk}b_{jk})^* -2^{n-3}( b_{jk} a_{jk})^{*} + 2^{n-3}(b_{jk} a_{jk}^*)^*) \\
&=\varphi(p_{n_*}(a_{jk},b_{jk},e_j,\ldots,e_j)) = p_{n_*}(\varphi(a_{jk}), \varphi(b_{jk}), q_j,\ldots,q_j)\\
&= \varphi(a_{jk})\varphi(b_{jk}) - \varphi(a_{jk})^{*}\varphi(b_{jk})^* \\
& \ - 2^{n-3} \varphi( b_{jk})^{*} \varphi(a_{jk})^{*} + 2^{n-3}\varphi(b_{jk})^{*} \varphi(a_{jk}^*)^*
\end{aligned}
\]
and then $\varphi(a_{jk}b_{jk}) = \varphi(a_{jk})\varphi(b_{jk})$.

%and then $\mathcal{I}(\varphi(a_{jk}b_{kk})) = \mathcal{I}(\varphi(a_{jk})\varphi(b_{kk}))$. On the other hand, using $ia_{jk}$ rather than $a_{jk}$ we obtain
%$\mathcal{R}(\varphi(a_{jk}b_{kk})) = \mathcal{R}(\varphi(a_{jk})\varphi(a_{kk}))$. It concludes that $\varphi(a_{jk}b_{kk}) = \varphi(a_{jk})\varphi(b_{kk})$.
%%%%%%%%%%%%%%

The others statements are proved in a similar way.
\end{proof}

Since alternative algebras are flexible, we have
\[
(x_{kj}, a_{jj}, b_{jj}) + (b_{jj}, a_{jj}, x_{kj})=0,
\]
for all $x_{kj} \in \A_{kj},$ $a_{jj},b_{jj} \in \A_{jj},$ for $k,j \in \{1,2\}.$

\begin{lemma}\label{lemam3}
If $a_{jj},b_{jj} \in \A_{jj}$, with $j \in \left\{1,2\right\}$, then $\varphi(a_{jj}b_{jj}) = \varphi(a_{jj})\varphi(b_{jj})$.
\end{lemma}
\begin{proof}
Let $x_{kj}$ be an element of $\A_{kj}$, with $j \neq k$. Using Lemma \ref{lemam2} we obtain
\begin{align*}
    \varphi(x_{kj})\varphi(a_{jj}b_{jj}) =& \varphi(x_{kj}a_{jj}b_{jj}) = \varphi ((x_{kj}a_{jj})b_{jj}) \\
    =& (\varphi(x_{kj})\varphi(a_{jj}))\varphi(b_{jj}) = \varphi(x_{kj})  (\varphi(a_{jj})\varphi(b_{jj}))\\
\end{align*}
that is,
$$
\varphi(x_{kj})(\varphi(a_{jj}b_{jj}) - \varphi(a_{jj})\varphi(b_{jj})) = 0.
$$
Now, by Lemma \ref{lemadireto}, $\varphi(x_{kj}) \in \A'_{kj}$ as well as $\varphi(a_{jj}b_{jj})$ and $\varphi(a_{jj})\varphi(b_{jj}) \in \A'_{jj}$.
Then, $(\varphi(e_k)\A')(\varphi(a_{jj}b_{jj}) - \varphi(a_{jj})\varphi(b_{jj})) = 0$, which implies that $\varphi(a_{jj}b_{jj}) = \varphi(a_{jj})\varphi(b_{jj})$ by Condition $(\ref{clubsuit})$ of Theorem~\ref{mainthm2}.
\end{proof}

\begin{dem2}
By additivity of $\varphi$ and Lemmas \ref{lemam2} and \ref{lemam3}, it follows  that $\varphi(ab) = \varphi(a)\varphi(b)$, for all $a$, $b \in \A$, this is, $\varphi$ preserves product as required.
\end{dem2}
%We complete the proof of Theorem \ref{mainthm2}.

\section{Corollaries} 

Now we present some consequences of our main results. %The first one provides the conjecture that appears in \cite{Ferco} to the case of multiplicative $*$-Lie-type maps:

\begin{corollary}   
Let $\A$ and $\A'$ be two alternative $*$-algebras with identities $1_\A$ and $1_{\A'}$, respectively, and $e_1$ and $e_2 = 1_{\A} - e_1$ nontrivial symmetric idempotents in $\A$. Let $\varphi: \A \rightarrow \A'$ be a complex scalar multiplication bijective unital map. Suppose that $\A$ satisfies
\begin{equation*}
    (e_j\A) x = \left\{0\right\} \mbox{ for any }  j \in \{ 1,2 \}\ \ \  \mbox{implies} \ \ \ x = 0.
\end{equation*}
%\begin{eqnarray*}
%  &&\left(\spadesuit\right) \ \ \  \ \ \  (e_j\A) x = \left\{0\right\} \ \ \  \mbox{implies} \ \ \ x = 0.	
%\end{eqnarray*}
 Even more, suppose that $\A'$ satisfies
\begin{equation*}
    (\varphi(e_j)\A') y = \left\{0\right\} \mbox{ for any }  j \in \{ 1,2 \}\ \ \  \mbox{implies} \ \ \ y = 0.
\end{equation*}
%\begin{eqnarray*}
% &&\left(\clubsuit\right) \ \ \  \ \ \  (\varphi(e_j)\A') y = \left\{0\right\} \ \ \  \mbox{implies} \ \ \ y = 0.
%\end{eqnarray*}
In this conditions, $\varphi: \A \rightarrow \A'$ is a multiplicative $*$-Lie $n$-map if and only if $\varphi$ is $*$-isomorphism.
\end{corollary}

It is easy to see that any prime alternative algebra satisfy Conditions \eqref{spadesuit}  and \eqref{clubsuit}, so we have the following result: 

\begin{corollary} \label{cor 3.2}
Let $\A$ and $\A'$ be two prime alternative $*$-algebras with identities $1_{\A}$ and $1_{\A'}$, respectively, and $e_1$ and $e_2 = 1_{\A} - e_1$ nontrivial symmetric idempotents in $\A$. In this condition, a complex scalar multiplication $\varphi: \A \rightarrow \A'$ is a bijective unital multiplicative $*$-Lie $n$-map if and only if $\varphi$ is $*$-isomorphism.
\end{corollary}

%\section{Application}

To finish we will give an application of the Corollary~\ref{cor 3.2}. A complete normed alternative complex $*$-algebra $A$ is called  an {\it alternative
$C^{*}$-algebra}  if it satisfies the condition: $\left\|a^{*}a\right\| = \left\|a\right\|^2$, for all elements $a \in A$. Alternative $C^{*}$-algebras are non-associative generalizations of
$C^{*}$-algebras and appear in various areas in Mathematics (see more details in the references \cite{Miguel1} and \cite{Miguel2}). 
An alternative $C^{*}$-algebra $A$ is called an 
{\it alternative $W^{*}$-algebra}
if it is a dual Banach space and a prime alternative $W^{*}$-algebra is called
{\it alternative $W^{*}$-factor}. It is well known that non-zero alternative $W^{*}$-algebras are unital.
% and that an alternative $W^{*}$-algebra $A$ is a factor if and only if its center is equal to $\mathbb{C}1_A$, where $1_A$ is the unit of $A$. 

\begin{corollary}
Let $\A$ and $\A'$ be two alternative $W^{*}$-factors. In this condition, a complex scalar multiplication $\varphi: \A \rightarrow \A'$  is a bijective unital multiplicative $*$-Lie $n$-map if and only if $\varphi$ is $*$-isomorphism.
\end{corollary}
%\begin{proof}
%Let $\mathcal{M}$ be a von Neumann algebra. It is shown in \cite{Bai} and \cite{Mie} that if a von Neumann algebra has no central summands of type $I_1$, then $\mathcal{M}$ satisfies the following assumption: 
%\begin{itemize}
%\item $e_j\mathcal{M}x = \left\{0\right\} \Rightarrow x = 0$.
%\end{itemize}
%Thus, by Theorem \ref{mainthm2} the corollary is true.
%\end{proof}

% To finish, $\mathcal{M}$ is a factor von Neumann algebra if its center only contains the scalar operators. It is well known that a factor von Neumann algebra is prime and then we have the following:

%\begin{corollary}
%Let $\mathcal{M}$ be a factor von Neumann algebra. Then a complex scalar multiplication $\varphi: \mathcal{M} \rightarrow \mathcal{M}$ is a bijective unital multiplicative $*$-Lie $n$-map if and only if $\varphi$ is a $*$-isomorphism.
%\end{corollary}


\begin{thebibliography}{99}

\bibitem{Bai} 
Z. Bai and S. Du, 
{\em Strong commutativity preserving maps on rings}, 
Rocky Mountain J. Math. \textbf{44} (2014), no. 3, 733-742.


\bibitem{brefos1} 
M. Bre\v sar and M. Fo\v sner, 
{\em On rings with involution equipped with some new product}, 
Publ. Math. Debrecen \textbf{57} (2000), no. 1-2, 121-134.


\bibitem{chang} 
Q. Chen and C. Li,  
{\em Additivity of Lie multiplicative mappings on rings}, 
Adv. in Math.(China),  \textbf{46} (2017),  no. 1, 82-90.


\bibitem{CuiLi} 
J.Cui and C. Li, 
{\em Maps preserving product $XY - YX^{*}$ on factor von Neumann algebras}, 
 Linear Algebra Appl. \textbf{431}, (2009), no. 5-7, 833-842.

\bibitem{LiLuFang} 
X. Fang, C. Li and F. Lu, 
{\em Nonlinear mappings preserving product $XY + YX^{*}$ on factor von Neumann algebras}, 
Linear Algebra Appl. \textbf{438}, (2013), no. 5, 2339-2345.


\bibitem{Ferco2} 
B. L. M. Ferreira and B. T. Costa, 
{\em $*$-Jordan-type maps on $C^{*}$-algebras}, 
Comm. Algebra, online version, DOI: 10.1080/00927872.2021.1937636, 2021.

\bibitem{Ferco} 
B. L. M. Ferreira and B. T. Costa, 
{\em $*$-Lie-Jordan-type maps on $C^{*}$-algebras}, 
 Bull. Iranian Math. Soc., online version, https://doi.org/10.1007/s41980-021-00609-4, 2021.

\bibitem{Fer3} 
B. L. M. Ferreira, 
{\em Multiplicative maps on triangular n-matrix rings}, 
International Journal of Mathematics, Game Theory and Algebra, \textbf{23}, p. 1-14, 2014.

\bibitem{Fer} 
J. C. M. Ferreira and B. L. M. Ferreira, 
{\em Additivity of $n$-multiplicative maps on alternative rings}, 
Comm. in Algebra  \textbf{44} (2016), no. 4, 1557-1568
 
 
\bibitem{bruth} 
R. N. Ferreira and B. L. M. Ferreira, 
{\em Jordan triple derivation on alternative rings}, 
Proyecciones \textbf{37} (2018), no. 1, 171-180.


\bibitem{Fer1} 
B. L. M. Ferreira, J. C. M.  Ferreira, H. Guzzo Jr., 
{\em Jordan maps on alternatives algebras}, 
JP J. Algebra Number Theory Appl. \textbf{31}  (2013), no. 2, 129-142.


\bibitem{Fer2} 
B. L. M. Ferreira, J. C. M.  Ferreira, H. Guzzo Jr., 
{\em Jordan triple elementary maps on alternative rings}, 
Extracta Math. \textbf{29}  (2014), no. 1-2, 1-18.


%\bibitem{Fer4} 
%B. L. M. Ferreira, J. C. M.  Ferreira, H. Guzzo Jr., 
%{\em Jordan triple maps of alternatives algebras}, 
%JP J. Algebra, Number Theory and Appl. \textbf{33} (2014), p. 25-33.

\bibitem{FerGur1} 
B. L. M. Ferreira and H. Guzzo Jr., 
{\em Lie maps on alternative rings}, 
Boll. Unione Mat. Ital., \textbf{13} (2020), 181-192.
 


\bibitem{brefos2} 
M. Fo\v sner, 
{\em Prime rings with involution equipped with some new product}, 
Southeast Asian Bull. Math. \textbf{26} (2002), no. 1, 27-31. 


\bibitem{He} I. R. Hentzel, E. Kleinfeld, H. F. Smith, 
{\it Alternative rings with idempotent}, 
J. Algebra \textbf{64} (1980), no. 2, 325-335.


\bibitem{Mie} 
W. S. Martindale III, 
{\it Lie isomorphisms of operator algebras} 
Pacific J. Math \textbf{38} (1971), 717-735.


\bibitem{Mart} 
W. S. Martindale III, 
{\it When are multiplicative mappings additive?} 
Proc. Amer. Math. Soc. \textbf{21} (1969), 695-698.


\bibitem{Miguel1} 
    Garc\'ia M. C., Palacios \'A. R., 
    Non-associative normed algebras. Vol. 1. The Vidav-Palmer and Gelfand-Naimark theorems. Encyclopedia of Mathematics and its Applications, 154. Cambridge University Press, Cambridge, 2014. xxii+712 pp.


\bibitem{Miguel2}  Garc\'ia M. C., Palacios \'A. R., 
    Non-associative normed algebras. Vol. 2. Representation theory and the Zel'manov approach. Encyclopedia of Mathematics and its Applications, 167. Cambridge University Press, Cambridge, 2018. xxvii+729 pp.
\end{thebibliography}
\end{document}